\documentclass{amsart} 
\usepackage{pb-diagram}
\usepackage[alphabetic,initials
]{amsrefs}
\author{Masaki Kameko}
\address{Department of Mathematical Sciences,
Shibaura Institute of Technology,
307 Minuma-ku Fukasaku, Saitama-City 337-8570, Japan}
\email{kameko@shibaura-it.ac.jp}
\thanks{This work was supported by JSPS KAKENHI Grant Numbers JP25400097 and JP17K05263.}
\title[Non-torsion classes]{Non-torsion non-algebraic classes in the Brown-Peterson tower}
\subjclass[2010]{14C25,19E15, 55T25}
\keywords{Brown-Peterson cohomology, classifying space, algebraic cycle, cycle map}
\newtheorem{theorem}{Theorem}[section]
\newtheorem{proposition}[theorem]{Proposition}
\newtheorem{lemma}[theorem]{Lemma}
\newtheorem{corollary}[theorem]{Corollary}

\theoremstyle{definition}

\newtheorem{example}[theorem]{Example}

\newcommand{\bp}[1]{BP\langle #1 \rangle}
\newcommand{\xc}{X(\mathbb{C})}
\newcommand{\Op}{\mathcal{O}}
\begin{document}

\begin{abstract}
Generalizing the classical work of Atiyah and Hirzebruch on non-algebraic classes, 
recently Quick proved the existence of torsion non-algebraic elements in the Brown-Peterson tower.
We construct  non-torsion non-algebraic elements in the Brown-Peterson tower for the prime number $2$. 
\end{abstract}

 \maketitle 

\section{Introduction}\label{section:1}

Let $X$ be a smooth complex algebraic variety. 
We denote  the space of complex points  of $X$ with the analytic topology by $\xc$. 
The Chow group $CH^iX$, the group generated by finite linear combinations of 
closed subvarieties of $X$ 
of codimension $i$ modulo rational equivalence,  is a fundamental object to study 
in algebraic geometry. 
The cycle map  from the Chow group $CH^iX$ to the singular cohomology 
$H^{2i}(\xc;\mathbb{Z})$
is also very important. 
We call an element in the cohomology of $\xc$ algebraic if it is in the image of the cycle map. 
One of the most outstanding open problems concerning the cycle map is the Hodge conjecture.
The Hodge conjecture asserts that  for a projective smooth complex algebraic variety $X$, the cycle map 
\[
CH^iX\otimes \mathbb{Q} \to \mbox{\{Hodge classes in $H^{2i}(\xc;\mathbb{Q})$\}}
\]
is surjective.
In \cite{totaro-1997}, Totaro defined the Chow group $CH^iBG$ of the classifying space of 
a complex linear algebraic group $G$ by approximating it by quasi-projective smooth complex algebraic 
varieties and studied the cycle map
\[
\mathrm{cl}\colon CH^iBG\to H^{2i}(BG;\mathbb{Z}).
\]
Let $p$ be a prime number and $(\mathbb{Z}/p)^3$  the elementary abelian $p$-group of rank $3$.
In retrospect, in \cite{atiyah-hirzebruch-1962}, Atiyah and Hirzebruch used the non-surjectivity of 
the cycle map 
\[
\mathrm{cl}\colon CH^2 B(\mathbb{Z}/p)^3 \to H^4(B(\mathbb{Z}/p)^3;\mathbb{Z}),
\]
that is, the existence of non-algebraic element in $H^{4}(B(\mathbb{Z}/p)^3;\mathbb{Z})$, 
to construct a projective smooth complex algebraic variety $X$  and a non-algebraic class in 
$H^{4}(\xc;\mathbb{Z})$, that  is a counterexample for the integral Hodge conjecture.

On the other hand, for each prime number $p$ and non-negative integer $m$, there exists the truncated 
Brown-Peterson cohomology $\bp{m}^*(\xc)$ whose coefficient ring is 
$\mathbb{Z}_{(p)}[v_1, \dots, v_m]$, where
$\deg v_i=-2p^i+2$.   Also there exist natural transformations
\[
BP^{*}(\xc)\stackrel{\rho_{m}}{\longrightarrow}\bp{m}^{*}(\xc)
\]
and 
\[
\bp{n}^{*}(\xc)
\stackrel{\rho_{m}^n}{\longrightarrow} \bp{m}^{*}(\xc)  \quad 
\mbox{($n\geq m$),}
\]
such that $\rho_m(v_i)=0$ for $i>m$, $\rho_m(v_i)=v_i$ for $i\leq m$ and 
$\rho_{m}=\rho_{m}^{n} \circ \rho_n$.
Moreover, there exist the motivic truncated Brown-Peterson cohomology 
$\bp{m}_{\mathcal{M}}^{*,*}(X)$ and the cycle map
\[
\mathrm{cl}\colon \bp{m}_{\mathcal{M}}^{2i,i}(X)\to \bp{m}^{2i}(\xc).
\]
In the case $m=0$, the motivic truncated Brown-Peterson cohomology 
$\bp{0}_{\mathcal{M}}^{*,*}(X)$ is nothing but the motivic cohomology 
$
H^{*,*}_{\mathcal{M}}(X;\mathbb{Z}_{(p)}),
$
the Chow group localized at $p$, say  $CH^iX_{(p)}$, is  the motivic cohomology 
$H^{2i,i}(X;\mathbb{Z}_{(p)})$ and the cycle map 
\[
CH^iX_{(p)}=H^{2i,i}_{\mathcal{M}}(X;\mathbb{Z}_{(p)})
=\bp{0}_{\mathcal{M}}^{2i,i}(X)
\to \bp{0}^{2i}(\xc)=H^{2i}(\xc;\mathbb{Z}_{(p)})
\]
is the classical cycle map. Recently,  in \cite{quick-2016}, for each prime number $p$ 
and positive integer $m$, Quick proved the existence of a  projective smooth complex algebraic variety 
$X$ and a non-algebraic element in $\bp{m}^*(\xc)$, that is an element not in the image of the cycle map 
$\mathrm{cl}\colon \bp{m}_{\mathcal{M}}^{2i,i}X\to \bp{m}^{2i}(\xc)$.
Quick uses  the classifying space of the elementary abelian $p$-group 
$(\mathbb{Z}/p)^{m+3}$. 
So, Quick generalized the work of Atiyah and Hirzebruch to the Brown-Peterson tower.
As examples of Atiyah and Hirzebruch, Quick's examples are also  torsion classes.

One might wonder the existence of non-torsion examples. 
Counterexamples for the integral Hodge conjecture modulo torsion were obtained by Koll\'ar. 
Koll\'ar's examples are explained, for example,  in \cite{soule-voisin-2005}. 
In  \cite{colliot-thelene-szamuely-2010}, Colliot-Th\'el\`ene\ and Szamuely used the topological approach 
of Atiyah and Hirzebruch to give counterexamples for the integral Tate conjecture over finite fields
and asked the existence of counterexamples  for the integral Tate conjecture modulo torsion where 
Koll\'ar's approach does not  work. 
In \cite{pirutka-yagita-2015},  counterexamples for the integral Hodge and Tate conjectures modulo torsion
 through the topological approach were  given by Pirutka and Yagita
 using exceptional groups $G_2$, $F_4$, $E_8$ which contain  elementary abelian $p$-subgroups 
$(\mathbb{Z}/p)^3$ for $p=2, 3, 5$, respectively.
In \cite{kameko-2015}, the author replaced the exceptional groups by 
$G_1=({SL_p}\times {SL_p})/\mu_p$ for all prime numbers $p$, where 
${SL_p}$ is the special linear group over the complex numbers and $\mu_p$ is
the center of ${SL_p}$  and $\mu_p$ acts on ${SL_p} \times {SL_p}$
diagonally.

In this paper, we construct non-torsion non-algebraic elements in the 
Brown-Peterson tower for $p=2$ and for all positive integers $m$ by replacing 
the elementary abelian $p$-group $(\mathbb{Z}/p)^{m+3}$ by
\[
G_{m+1}={SL_p}^{m+2}/\mu_p,
\] 
the central quotient of product of special linear groups ${SL_p}$ where $\mu_p$ is the cyclic group of order $p$ and it acts on ${SL_p}^{m+2}$
diagonally. 
We use the mod $p$ cohomology of the classifying space of $G_n$
to compute the Atiyah-Hirzebruch spectral sequence converging to 
the truncated Brown-Peterson cohomology $\bp{m}^{*}(BG_n)$.
It is easy to  compute the mod $p$ cohomology of $G_n$ 
for $p=2$ but it is not  easy for odd prime numbers $p$. 
Hence our result is limited to the case $p=2$ although it is plausible 
that the similar results hold for all prime numbers $p$. 

%:  theorem:1.1

\begin{theorem} \label{theorem:1.1}
Suppose that $p=2$.  For each non-negative integer $m$, there exists a 
non-torsion element $u$ in the integral cohomology of $BG_{m+1}$ such that
the element  $u$ lifts up to 
$
BP\langle m \rangle^{*}(BG_{m+1})
$ 
but not to 
$
BP\langle m+1 \rangle^{*}(BG_{m+1}).
$
\end{theorem}

To state our result, Corollary~\ref{corollary:1.3}, obtained from Theorem~\ref{theorem:1.1}, 
let $\mathbb{G}_{m}$ be the multiplicative group of the non-zero complex numbers.
For a reductive complex linear algebraic group $G$,
we may approximate the classifying space 
$
B(\mathbb{G}_{m} \times G)
$
by projective smooth complex algebraic varieties. That is, for each positive integer $k$, there exists a projective smooth complex algebraic variety $X$ and a morphism of algebraic varieties 
that induces a continuous map $f\colon \xc \to B(\mathbb{G}_m\times G)$ which is a 
$k$-equivalence. In particular, 
\[
f^{*}\colon H^{i}(BG;\mathbb{Z}_{(p)}) \to H^{i}(\xc ;\mathbb{Z}_{(p)})
\]
is an isomorphism for $i<k$.
For the approximation of the above classifying space by projective smooth complex algebraic varieties, 
we refer the reader to  Pirutka and Yagita \cite{pirutka-yagita-2015} and Antieau \cite{antieau-2016} 
for the detail. 
For the motivic Brown-Peterson cohomology $\bp{m}_{\mathcal{M}}^{2i,i}(X)$, we refer the reader to Quick
\cite{quick-2016} for the detail. 
With these preparations, we have Corollary~\ref{corollary:1.3}  on the existence of non-torsion non-algebraic classes
of projective smooth complex algebraic varieties
from Theorem~\ref{theorem:1.1}.
But, before we state Corollary~\ref{corollary:1.3}, we show that, rationally,  all cohomology classes obtained from the classifying space of a complex linear algebraic group are algebraic.

%:  proposition:1.2

\begin{proposition}\label{proposition:1.2}
Let $G$ be a 
complex linear algebraic group. Let $X$ be a projective smooth complex algebraic variety. Suppose that $f\colon \xc  \to BG$ is the map induced by
a morphism of complex algebraic varieties from $X$ to $BG$ and $v'\in \bp{m}^{2i}(BG)$. Then, there exists a nonzero integer $\alpha$ such that $\alpha f^{*}(v')$ is algebraic in the sense that 
$\alpha f^{*}(v')$ is in the image of the cycle map
\[
\mathrm{cl} \colon \bp{m}^{2i,i}_{\mathcal{M}}(X)\to \bp{m}^{2i}(\xc ).
\]
\end{proposition}
 
 \begin{proof}
First, we recall that 
\[
BP^{2i,i}_{\mathcal{M}}(BGL_\ell) \to BP^{2i}(BGL_\ell)
\]
is an isomorphism \cite[Proposition 9.1]{yagita-2005}
and 
\[
\rho_m\colon BP^{2i}(BGL_\ell)\to \bp{m}^{2i}(BGL_\ell)
\]
is surjective for the complex general linear group $GL_\ell$. Therefore, Chern classes 
in $\bp{m}^{*}(BG)$ of complex representations of $G$ are 
algebraic.
Next, we recall that 
 the rational cohomology $H^{*}(BG;\mathbb{Q})$ is generated by Chern classes of complex representations of $G$
 \cite[Proof of Theorem 1]{landweber-1972}.
The rational cohomology of $BG$ is a polynomial algebra generated by
even degree elements $y_1, \dots, y_n$. The Atiyah-Hirzebruch spectral sequence for
$\bp{m}^{*}(BG)\otimes \mathbb{Q}$ collapses at the $E_2$-level and 
\[
\bp{m}^{*}(BG)\otimes \mathbb{Q}=\bp{m}^{*}[[ y_1, \dots, y_n]]\otimes \mathbb{Q}.
\]
We may assume $y_1, \dots, y_n$ are Chern classes of
complex representations of $G$, so that $y_1, \dots, y_n$  are elements in 
$
\bp{m}^{*}(BG)
$ and they are algebraic.
Let 
\[
j \colon \bp{m}^{2i}(BG) \to \bp{m}^{2i}(BG)\otimes \mathbb{Q}
\]
be the homomorphism induced by the inclusion map $\mathbb{Z}_{(p)}\to \mathbb{Q}$.
Then, 
\[
j(v')=\sum_{I, J} \dfrac{\beta_{IJ}}{\alpha_{IJ}} v_{I}y_{J}
\]
where $v_I$ ranges over monomials in $v_1, \dots, v_m$, 
$y_J$ ranges over monomials in $y_1, \dots, y_n$, under the restriction
\[
\deg v_I+\deg y_{J}=2i,
\]
$\alpha_{IJ}$ are nonzero integers and $\beta_{IJ}$ are integers.
The element $j(v')$ is an infinite linear combination of such $v_I y_J$.
Let $\dim \xc $ be the dimension of $\xc $ as a differentiable manifold.
Since there exist only finite number of $v_Iy_J$ such that
\[
\deg v_I+\deg y_J=2i, \quad \deg y_J \leq \dim \xc ,
\]
we may choose a common multiplier $\alpha''$ of $\alpha_{IJ}$ of such $(I,J)$.
Let us define the element $v''$ in $\bp{m}^{2i}(BG)$ by
\[
v''=\sum_{I,J}  \dfrac{\alpha'' \beta_{IJ}}{\alpha_{IJ}} v_I y_{J}
\]
where $\deg v_I+\deg y_J=2i$, $\deg y_J\leq  \dim \xc $.
Then, $v''$ is algebraic and 
\[
j (\alpha'' v'-v'')=\sum_{I,J} \dfrac{\alpha''\beta_{IJ}}{\alpha_{IJ}} v_I y_J
\]
where $\deg v_I+\deg y_J=2i$, $\deg y_J> \dim \xc $.
Therefore, in $\bp{m}^{2i}(\xc )\otimes \mathbb{Q}$, we have 
\[
j(f^{*}(\alpha'' v'-v''))=0
\]
On the other hand, since $\xc $ has the homotopy type of a finite complex, 
the kernel of 
\[
j\colon \bp{m}^{2i}(\xc ) \to \bp{m}^{2i}(\xc ) \otimes \mathbb{Q}
\]
is torsion and there exists a nonzero integer $\alpha'$ such that 
\[
\alpha'(\alpha'' f^{*}(v')-f^{*}(v''))=0.
\]
So, putting $\alpha=\alpha'\alpha''$, we have that
\[
\alpha f^{*}(v')=\alpha' f^{*}(v'')
\]
and it is algebraic.
 \end{proof}
 
 Now, we state and prove Corollary~\ref{corollary:1.3} on the existence of non-torsion, non-algebraic classes assuming Theorem~\ref{theorem:1.1}.
 
%:  corollary:1.3

\begin{corollary} \label{corollary:1.3}
Suppose that $p=2$. For each non-negative integer $m$, there exists a  projective smooth 
complex algebraic variety $X$ 
and an element $v$ in $\bp{m}^{*}(X)$
such that 
\begin{itemize}
\item[(1)] $v$ is non-algebraic, that is, not in the image of the cycle map 
\[
\mathrm{cl}\colon \bp{m}_{\mathcal{M}}^{2i,i}(X) \to \bp{m}^{2i}(\xc)
, 
\]
\item[(2)] there exists a nonzero integer $\alpha$ such that the scalar multiple $\alpha v$ of $v$ is algebraic, 
and 
\item[(3)] $\rho^{m}_0(v)$ is a non-torsion class in $H^{2i}(\xc;\mathbb{Z}_{(p)})$,
\end{itemize}
where $2i=\deg v$.
\end{corollary}

\begin{proof}
Let $G_{m+1}$ be the algebraic group $G_{m+1}$ in Theorem~\ref{theorem:1.1} and 
 $u\in H^{2i}(BG_{m+1};\mathbb{Z}_{(p)})$ be the element $u$  in Theorem~\ref{theorem:1.1}.
For the sake of notational simplicity, let $G=\mathbb{G}_m\times G_{m+1}$

By the K\"unneth theorem, we have
\[
H^{2i}(BG;\mathbb{Z}_{(p)})=
\bigoplus_{k} H^{k}(B\mathbb{G}_m;\mathbb{Z}_{(p)}) \otimes H^{2i-k}(BG_{m+1};\mathbb{Z}_{(p)})
\]
Let us denote by $1\otimes u$ the element 
 in $H^{2i}(BG;\mathbb{Z}_{(p)})$ 
 corresponding  to the element $1\otimes u$
in 
\[
H^{0}(B\mathbb{G}_m;\mathbb{Z}_{(p)}) \otimes H^{2i}(BG_{m+1};\mathbb{Z}_{(p)}).
\]
In the Atiyah-Hirzebruch spectral sequence for $\bp{m}^{*}(BG)$, 
\[
(1\otimes u) \otimes 1 \in E_2^{2i,0}
\]
is a permanent cycle. Let $v'$ be an element in $\bp{m}^{2i}(BG)$ such that 
\[
\rho^m_0(v')=1\otimes u. 
\]
In the Atiyah-Hirzebruch spectral sequence for $\bp{m+1}^{*}(BG)$, 
for some $r$, 
\[
d_r((1\otimes u) \otimes 1)\not=0 \in E_r^{2i+r,1-r}.
\]

Choose a projective smooth complex algebraic variety  $X$ and a morphism from $X$ to $BG$, so that
\[
f^*\colon H^{k}(BG;\mathbb{Z}_{(p)})\to H^{k}(\xc ;\mathbb{Z}_{(p)})
\]
is an isomorphism for $k\leq 2i+r$, where $f\colon \xc \to BG$ is the map induced by the morphism of algebraic varieties. Then, 
in the Atiyah-Hirzebruch spectral sequence for $\bp{m+1}^{*}(\xc )$, 
\[
d_r(f^{*}(1\otimes u)\otimes 1)\not=0.
\]
Therefore, $f^{*}(1\otimes u)$ is a non-torsion class and $f^{*}(1\otimes u)$ is in the image of $\rho^m_0$ but  not in the image of $\rho^{m+1}_0$.

Let $v=f^{*}(v')$. 
Then, $\rho^m_0(v)=f^{*}(1\otimes u)$. So, 
the element $v$ satisfies  the condition (3).
By Proposition~\ref{proposition:1.2}, the element $v$ 
satisfies the condition (2). 
As for the condition (1), 
we consider the following commutative diagram:
\[
\begin{diagram}
\node{H^{2i}(\xc ;\mathbb{Z}_{(p)})} \arrow{e,t}{=} 
\node{\bp{0}^{2i}(\xc )} 
\node{\bp{m}^{2i}(\xc )}\arrow{w,t}{\rho^m_0}
\\
\node{CH^{i}X_{(p)}} \arrow{e,t}{=} \arrow{n,l}{\mathrm{cl}}
\node{\bp{0}_{\mathcal{M}}^{2i,i}(X)}  \arrow{n,l}{\mathrm{cl}}
\node{\bp{m}_{\mathcal{M}}^{2i,i}(X).} \arrow{n,l}{\mathrm{cl}}
\arrow{w,t}{\rho_0^m}
\end{diagram}
\]
The image of the composition of $\rho^m_0$ and the cycle map 
\[
\mathrm{cl} \colon \bp{m}_{\mathcal{M}}^{2i,i}(X )\to \bp{m}^{2i}(\xc ),
\]
is a subset of the image of the cycle map 
\[
\mathrm{cl} \colon CH^{i}(X)\to H^{2i}(\xc ;\mathbb{Z}_{(p)}). 
\]
However, Totaro proved that the image of the cycle map is in the image of $\rho_0$. Thus, since $f^{*}(1\otimes u)$ does not lift up to 
$\bp{m+1}^{2i}(\xc )$, the element $v$  is not in the image of the cycle map
and it satisfies the condition (1).
\end{proof}

The Brown-Peterson cohomology plays an interesting role  in both algebraic geometry 
and algebraic topology, especially in conjunction with the study of the cycle map. 
In \cite{totaro-1997}, Totaro showed that the cycle map
\[
\mathrm{cl}\colon CH^{i} X _{(p)}  \to H^{2i}(\xc;\mathbb{Z}_{(p)})
\]
factors through $BP^{*}(\xc)\otimes_{BP^*} \mathbb{Z}_{(p)}$.
He  also conjectured that
the cycle map 
\[
CH^i BG _{(p)} \to (BP^{*}(BG)\otimes_{BP^*}\mathbb{Z}_{(p)})^{2i}
\]
is an isomorphism for a complex linear algebraic group $G$ when the Brown-Peterson cohomology of
$BG$ has no nonzero odd degree elements.
The Brown-Peterson cohomology had been studied in algebraic topology long before Totaro's work.
It is also conjectured in \cite{kono-yagita-1993} that the Brown-Peterson cohomology of the classifying space 
of a compact Lie group has no nonzero odd degree element.
This conjecture  also remains as an open problem and 
it is a precursor of Totaro's conjecture. Note that the maximal compact subgroup of a reductive complex 
linear algebraic group is a compact Lie group and they are homotopy equivalent to each other.
So, the truncated Brown-Peterson cohomology theories of classifying spaces of these groups are the same to
each other.
 In particular, $G_n$'s maximal compact subgroup is $SU(p)^{n+1}/\mu_p$, 
 where $SU(p)$ is the special unitary group.

In this paper, through the computation of the Atiyah-Hirzebruch spectral sequence, 
we give a new family of compact Lie groups $SU(2)^{n+1}/\mu_2$
for which the Brown-Peterson cohomology of the classifying space
 has no nonzero odd degree element.

%:  theorem:1.4

\begin{theorem} \label{theorem:1.4}
Suppose that $p=2$, $m$ is a non-negative integer and $n$ is  a positive integer. The truncated 
Brown-Peterson cohomology
$\bp{m}^*BG_n$ has no nonzero odd degree elements  if and only if $m\geq n+1$.
Moreover,  
 the Brown-Peterson cohomology  $BP^*(BG_n)$ has no nonzero odd degree elements
 and 
the natural transformation $\rho_m:BP^*(BG_n)
\to \bp{m}^{*}(BG_n)$ is surjective if and only if $m\geq n+1$.
\end{theorem}

For  $n=1$, $G_1=SO(4)$. The Brown-Peterson cohomology of $BG_1=BSO(4)$ is computed  by Kono and Yagita~\cite[Theorem 5.5]{kono-yagita-1993}. 

%Although the computation of the Chow group of $BG_n$ is important and interesting, 
%we leave it as  an open problem.
This paper is organized as follows: In Section~\ref{section:2}, we compute the mod $2$ cohomology 
of $BG_n$ and the action of the Milnor operations.
In Section~\ref{section:3}, we define a homomorphism $\phi_r$ and prove Lemmas~\ref{lemma:3.2} 
and \ref{lemma:3.3}. We need the homomorphism $\phi_r$ to state our main result. 
We use these lemmas in Sections~\ref{section:4} and \ref{section:5}.
In Section~\ref{section:4}, 
we write down the Atiyah-Hirzebruch spectral sequence for $\bp{m}^{*}(BG_n)$ as 
Propositions~\ref{proposition:4.1} and \ref{proposition:4.2}. 
In Section~\ref{section:5}, we prove the results stated in Section~\ref{section:4}.

\section{The mod $2$ cohomology ring}\label{section:2}

In this section, we compute the mod $2$ cohomology ring of $BG_n$ and the action of 
the Milnor operations on it.

%:  BG_n

First, we compute the mod $2$ cohomology of $BG_n$ for $n\geq 1$.
We consider the projection $\pi$ obtained by removing the $(n+1)$-th factor, that is, 
\[
\pi \colon {SL_2}^{n+1}\to {SL_2}^{n}, \quad  \pi(m_0, m_1, \dots, m_n)=(m_0, m_1, \dots, m_{n-1}).
\]
It induces a fibre sequence
\[
B{SL_2} \longrightarrow  BG_n \stackrel{\pi}{\longrightarrow}
BG_{n-1}. 
\]
It is clear that $G_0$, $BG_0$ are homotopy equivalent to $SO(3)$, $BSO(3)$, respectively. Therefore, we have
\[
H^{*}(BG_0;\mathbb{Z}/2)=\mathbb{Z}/2[x_2, x_3],
\]
where $x_2$ is the generator of $H^{2}(BSO(3);\mathbb{Z}/2)$, 
$x_3=Q_0 x_2$ and $Q_0$ is the mod $2$ Bockstein homomorphism.
Let us recall the mod $2$ cohomology of $B{SL_2}$. 
The mod $2$ cohomology of $B{SL_2}$ is a polynomial algebra generated by the single generator 
$c_2$ of degree $4$. 
The Leray-Serre spectral sequence associated with the above fibre sequence
has the $E_2$-term isomorphic to 
\[
H^{*}(BG_{n-1};\mathbb{Z}/2) \otimes H^{*}(B{SL_2};\mathbb{Z}/2).
\]
The $E_2$-term is generated by $1\otimes c_2$ as an algebra over $E_2^{*,0}$.
So, in order to prove that the spectral sequence collapses at the $E_2$-level, it suffices to show that 
$1 \otimes c_2$ is a permanent cycle. For dimensional reasons, $d_r(1 \otimes c_2)=0$ except for $r=5$
and $d_r(1\otimes c_2)\in E_5^{5,0}$.
Let us consider the map $\sigma \colon G_{n-1}\to G_n$ induced by  
$\sigma \colon {SL_2}^n \to {SL_2}^{n+1}$, given by
\[
\sigma(m_0, \dots, m_{n-1})=(m_0,  \dots, m_{n-1}, m_{n-1}).
\]
Then, 
$\pi \circ \sigma$ is the identity map. 
Hence, 
all elements in  $E_2^{*,0}$ are permanent cycles. Therefore,  we have
$d_{5}(1\otimes c_{2})=0$ and so $1\otimes c_{2}$ is also a permanent cycle.
So, the spectral sequence collapses at the $E_2$-level and  we have 
\[
H^{*}(BG_n;\mathbb{Z}/2)=
H^{*}(BG_{n-1};\mathbb{Z}/2)\otimes H^{*}(B{SL_2};\mathbb{Z}/2)
\]
 as an algebra over $H^{*}(BG_{n-1};\mathbb{Z}/2).$
In particular, for $n=1$, we have
\[
H^{*}(BG_1;\mathbb{Z}/2)=\mathbb{Z}/2[x_2,x_3, x_4]
\]
for some $x_{4}$. We make the choice of $x_{4}$ in due course.
Inductively, we have 
\[
H^{*}(BG_n;\mathbb{Z}/2)=\mathbb{Z}/2[x_2,x_3, x_{41}, \dots, x_{4n}]
\]
for some degree $4$ elements $x_{41}, \dots, x_{4n}$. 
For $i=1, \dots, n$, 
let $\pi_i \colon G_n\to G_1$ be the map induced by the projection of ${SL_2}^{n+1}$ to ${SL_2}^{2}$ given by
\[
(m_0, \dots, m_n) \mapsto (m_0,  m_{i}).
\]
Then, we have the following proposition.

%:  proposition:2.1

\begin{proposition}
\label{proposition:2.1}
The mod $2$ cohomology of $BG_n$ is a polynomial algebra over $\mathbb{Z}/2$, 
\[
\mathbb{Z}/2[x_2, x_3, x_{41}, \dots, x_{4n}],
\]
where $\deg x_2=2$, $\deg x_3=3$, $\deg x_{4i}=4$.
The generator $x_{4i}$ is $\pi_i^*(x_4)$.
\end{proposition}

Next, we consider the subgroup $(\mathbb{Z}/2)^{n+2}$ of $G_n$.
Let us consider elements 
\[
I=\begin{pmatrix} 1 & 0 \\ 0 & 1 \end{pmatrix}, \;
\xi=\begin{pmatrix} -1 & 0 \\ 0 & -1 \end{pmatrix}, \; 
\eta=\begin{pmatrix} i & 0 \\ 0 & -i \end{pmatrix}, \; 
\zeta=\begin{pmatrix} 0 & i \\ i & 0 \end{pmatrix}, 
\]
in ${SL_2}$. 
Let 
\[
\Delta \colon {SL_2} \to {SL_2}^{n+1}
\]
be the diagonal map, that is, $\Delta(m)=(m,\dots, m)$.
The subgroup $\mu_2$ of the center of ${SL_2}^{n+1}$ is the subgroup 
generated by $\Delta(\xi)=(\xi, \dots, \xi)$.
For $i=0,\dots, n$, let us write
\[
\Gamma_i \colon {SL_2}\to {SL_2}^{n+1}
\]
for the inclusion map into the $i$-th factor  such that
\[
\Gamma_i(m)=(m_0, \dots, m_n), 
\]
where $m_i=m$, $m_j=I$ for $j\not=i$.
Let 
\[
\pi'\colon {SL_2}^{n+1}\to {SL_2}^{n+1}/\mu_2=G_n
\]
be the obvious projection.
Then, the elements 
\[
\pi'(\Delta(\eta)), \; \pi'(\Delta(\zeta)), \; \pi'(\Gamma_i(\xi)), \quad \mbox{($i=1, \dots, n$)},
\]
generate a subgroup isomorphic to the elementary abelian $2$-group of rank $n+2$.
We denote it by $(\mathbb{Z}/2)^{n+2}$
and we denote 
the inclusion map of $(\mathbb{Z}/2)^{n+2}$ into $G_{n}$ by
\[
\iota_n \colon (\mathbb{Z}/2)^{n+2}\to G_n.
\]

Now, we recall some facts on   the induced homomorphism 
\[
\iota_1^*\colon H^{*}(BG_1;\mathbb{Z}/2)
\to H^{*}(B(\mathbb{Z}/2)^3;\mathbb{Z}/2).
\]
We denote by 
$s_1, s_2, t$  
generators of $H^1(B(\mathbb{Z}/2)^{3};\mathbb{Z}/2)$ corresponding to the generators 
$\pi'(\Delta(\eta)), \pi'(\Delta(\zeta)), \pi'(\Gamma_{1}(\xi))$.
As above, let $x_2$ be the generator of $H^2(BG_1;\mathbb{Z}/2)=\mathbb{Z}/2$.
Let $x_3=Q_0(x_2)$. 
Then, 
\[
\iota_1^*(x_2)=s_1^2+s_1s_2+s_2^2, \quad 
\iota_1^*(x_3)=s_1^2s_2+s_1s_2^2, 
\]
respectively.
The image of $\iota_1^*$ is generated by the above two elements and
\[
t^4+(s_1^2+s_1s_2+s_2^2)t^2+(s_1^2s_2+s_1s_2^2)t,
\]
as an algebra over $\mathbb{Z}/2$.
See \cite[Section 6]{kameko-2015}, if needed,  
for the details of the image  of generators of $H^{*}(BG_1;\mathbb{Z}/2)$ by the induced homomorphism 
$\iota_1^*$.
We choose the generator $x_4$ of $H^4(BG_1;\mathbb{Z}/2)$, so that its image 
$\iota_1^*(x_4)$ is the above element of degree $4$.
Then, it is easy to see that the following proposition holds.

%:  proposition:2.2

\begin{proposition}
\label{proposition:2.2}
The induced homomorphism 
\[
\iota_n^*\colon H^{*}(BG_n;\mathbb{Z}/2)\to H^{*}(B(\mathbb{Z}/2)^{n+2};\mathbb{Z}/2)
\]
is injective.
\end{proposition}

%:  proposition:2.3

Finally, we close this section by determining the action of Milnor operations on $x_3$, $x_4$. 
It plays an important role 
in the computation of the differentials of the Atiyah-Hirzebruch spectral sequence 
for $\bp{m}^*(BG_n)$, $n\geq 2$. 
Let us denote by $Q_j$ the Milnor operation of degree $2^{j+1}-1$.

\begin{proposition}
\label{proposition:2.3}
For $j\geq 1$, $0 \leq k \leq j-1$, 
there exist polynomials  $\alpha_{jk}$ in $x_2^2,x_3^2$ such that
\[
{Q}_j (x_3)=  \alpha_{j0} x_3^2, \quad 
{Q}_j (x_4)= x_3 \left( \sum_{k=0}^{j-1}  \alpha_{jk} x_4^{2^k}\right)
\]
in $H^{*}(BG_1;\mathbb{Z}/2)=\mathbb{Z}/2[x_2,x_3,x_4]$.
\end{proposition}

\begin{proof}
For the sake of notational simplicity, by abuse of notation, 
we denote $\iota_1^*(x_i)$ by $x_i$.
By direct calculation, 
we have 
\begin{align*}
Q_0(x_3)&=0, & Q_1(x_3)&= x_3^2, & Q_2(x_3)&=x_2^2 x_3^2, \\
Q_0(x_4)&=0, & Q_1(x_4)&=x_3x_4, & Q_2(x_4)&= x_3x_4^2+x_2^2x_3x_4.
\end{align*}
We define $\alpha_{10}=0$, $\alpha_{20}=x_2^2$, $\alpha_{21}=1$
and the proposition holds for $i=1, 2$.

For $j\geq 2$, let
\[
\Op_j=Q_j+ x_2^{2^{j-1}}Q_{j-1}+x_3^{2^{j-1}}Q_{j-2}.
\]
It is clear that $\Op_j$ acts as a derivation on both 
$H^{*}(BG_1;\mathbb{Z}/2)$ and $H^{*}(B(\mathbb{Z}/2)^3;\mathbb{Z}/2)$
and 
\[
\Op_2(s_1)=\Op_2(s_2)=0, \quad \Op_2(t)=x_4^2.
\]
It is also easy to verify that 
\[
\Op_j(s_1)=(\Op_2(s_1))^{2^{j-2}}, \quad 
\Op_j(s_2)=(\Op_2(s_2))^{2^{j-2}}, \quad 
\Op_j(t)=(\Op_2(t))^{2^{j-2}},
\]
for $j\geq 2$.
Hence, for $j\geq 2$, in $H^{*}(B(\mathbb{Z}/2)^3;\mathbb{Z}/2)$, we have 
\[
\Op_j(s_1)=\Op_j(s_2)=0 \quad \mbox{and} \quad 
\Op_j (t)=x_4^{2^{j-1}}.
\]
Since
$
\Op_j
$
is a derivation, we have
\[
\Op_j(x_3)=0, \quad \Op_j(x_4)=x_3 \Op_j(t)=x_3 x_4^{2^{j-1}}.
\]
Therefore, we have 
\begin{align*}
Q_j (x_3)&= x_2^{2^{j-1}} Q_{j-1}x_3+x_3^{2^{j-1}} Q_{j-2} x_3, 
\\
Q_j (x_4) & = x_3 x_{4}^{2^{j-1}}+x_2^{2^{j-1}} Q_{j-1}x_4+x_3^{2^{j-1}} Q_{j-2} x_4.
\end{align*}
We prove the proposition by  induction on $j$.
Assume that $j\geq 3$ and  the proposition holds for $j-1$ and $j-2$.
Then, 
we have
\begin{align*}
Q_j(x_3)&=
x_{2}^{2^{j-1}} Q_{j-1}(x_3)+ 
x_{3}^{2^{j-1}} Q_{j-2}(x_3)
\\
&= x_3 \left( 
\alpha_{j-1,0} x_2^{2^{j-1}}  + \alpha_{j-2,0}x_3^{2^{j-1}} \right),
\end{align*}
and
\begin{align*}
Q_j(x_4)&=x_3 x_4^{2^{j-1}} + 
x_{2}^{2^{j-1}} Q_{j-1}(x_4)+ 
x_{3}^{2^{j-1}} Q_{j-2}(x_4)
\\
&
=
x_3 \left( x_4^{2^{j-1}}+
\alpha_{j-1,j-2} x_2^{2^{j-1}} x_4^{2^{j-2}} + 
\sum_{k=0}^{j-3} (\alpha_{j-1,k} x_2^{2^{j-1}} +\alpha_{j-2,k} x_3^{2^{j-1}}) x_4^{2^k}  \right).
\end{align*}
Hence, we have 
\begin{align*}
\alpha_{j,j-1}&=1, \\
\alpha_{j,j-2}& =\alpha_{j-1,j-2} x_2^{2^{j-1}},  \\
\alpha_{j,k}&=\alpha_{j-1,k} x_2^{2^{j-1}}  + \alpha_{j-2,k}x_3^{2^{j-1}}
\end{align*}
 for $0\leq k <j-2$
as desired.
\end{proof}

\section{Lemmas} \label{section:3}

Let $M_n=\mathbb{Z}/2[x_2^2, x_3, x_{41}, \dots, x_{4n}]$
 and we regard $M_n$ as a subalgebra of the mod $2$ cohomology ring of 
 $BG_n$ in Proposition~\ref{proposition:2.1}.
 It is  closed under the action of Milnor operations $Q_j$.
In this section, we define a homomorphism $\phi_r\colon  M_n \to M_n$.
We use this $\phi_r$ in the statement of our main result, Proposition~\ref{proposition:4.2}. 
Then, we prove Lemmas~\ref{lemma:3.2} and \ref{lemma:3.3} on this homomorphism $\phi_r$. 
We use Lemma~\ref{lemma:3.2} in the proof of Theorem~\ref{theorem:1.4} in Section~\ref{section:4} and 
Lemma~\ref{lemma:3.3} in the proof of Proposition~\ref{proposition:4.2} in Section~\ref{section:5},
respectively.
We do not use Lemmas~\ref{lemma:3.2} and \ref{lemma:3.3} elsewhere.

We set up some notations.
For a ring $R$ and  a finite set $\{ y\}$, we write $R\{ y\}$ for the free $R$-module with the basis $\{ y\}$.
Let 
\begin{align*}
M_n&=\mathbb{Z}/2[x_2^2, x_3, x_{41}, \dots, x_{4n}], 
\\
L_n&=\mathbb{Z}/2[x_2^2, x_3, x_{41}^2, \dots, x_{4n}^2], 
\end{align*}
For a finite set $I=\{ i_1, \dots, i_{\ell} \}$ of positive integers such that $1\leq i_1 < \cdots < i_{\ell} \leq n$, 
we write $x_I$ for $x_{4 i_1}\cdots x_{4i_\ell}$ and $\ell(I)$ for $\ell\geq 1$.
It is clear that $M_n$ is a free $L_n$-module with the basis $\{ 1, x_{I} \}$.
Let us denote by $M_{n,i}$ the submodule spanned by
$\{ x_{I}\;|\; \ell(I)=i\}$ for $i\geq1$. We write  $M_{n,0}$ for $L_n\{ 1\}\subset M_{n}$.
For $j\geq 2$, we define  derivations
$\partial_i$ by 
\begin{align*}
\partial_j(x_3) &=0, \\
\partial_j(x_{4i})&= \sum_{k=1}^{j-1} \alpha_{jk } x_{4i}^{2^k},
\end{align*}
where the coefficients $\alpha_{jk}$ are those in 
Proposition~\ref{proposition:2.3}.
From this definition, it is clear that 
\[
Q_j(x_3x_{4i})=x_3^2 \partial_j(x_{4i}), \quad Q_j(x_{4i_{1}}x_{4i_{2}})=x_3 \partial_j(x_{4i_{1}}x_{4i_{2}}).
\]
Moreover, $\partial_j$ is an $L_n$-module homomorphism from $M_{n,i+1}$ to $M_{n,i}$.
 We define an $L_n$-module homomorphism $\phi_r:M_{n,i+r-1}\to M_{n, i}$ by 
 \[
 \phi_r(x)=\partial_2\cdots \partial_r(x) \in M_{n, i},
 \]
where $x\in M_{n,i+r-1}$. The following proposition is immediate from the definition of $\phi_r$
and the fact that 
\[
Q_j(x_3x_{4i})=x_3^2\partial_j(x_{4i}), \quad Q_j(x_{4i_{1}}x_{4i_{2}})=x_3 \partial_j(x_{4i_{1}}x_{4i_{2}}).
\]

%:  proposition:3.1

\begin{proposition}
\label{proposition:3.1}
Suppose that $\ell(I)-r+1+s$ is even and $s\geq 0$.
Then, we have
\[
Q_{r+1} (\phi_r(x_I)x_3^s)=\phi_{r+1}(x_I)x_3^{s+1}.
\]
\end{proposition}

Now, we state and prove Lemma~\ref{ lemma:3.2}. We use Lemma~\ref{ lemma:3.2} in 
Section~\ref{section:4} in the proof of Theorem~\ref{theorem:1.4}.

%:  lemma:3.2

\begin{lemma}
\label{lemma:3.2}
Let us define the polynomial $e_{k}(t_1, \dots, t_{k})$ by 
\[
e_k(t_1, \dots, t_{k})= \begin{vmatrix}
t_1 &  \cdots & t_k \\
t_1^2 & \cdots & t_k^2 \\
\vdots & & \vdots \\
t_1^{2^{k-1}} & \cdots & t_k^{2^{k-1}}
\end{vmatrix}.
\]
Then, in $M_n$, for a finite set 
 $I=\{ i_1, \dots, i_\ell\}$ of positive integers such that $1\leq i_1<\cdots <i_\ell\leq n$, 
we have
\label{ lemma:3.2}
\begin{align*}
\phi_\ell (x_I) &=e_\ell(x_{4i_1}, \dots, x_{4i_{\ell}}) \not=0,
\\
\phi_{\ell+1}(x_I)&=e_\ell(x_{4i_1}, \dots, x_{4i_{\ell}})^2 \not=0,
\\
\phi_{\ell+2} (x_I)&=0.
\end{align*}
\end{lemma}

\begin{proof}
Immediately from the definition of $\partial_j$, we  have $\partial_{j_1}\partial_{j_2} (x_{4i})=0$. Let $S_{\ell}$ be the set of permutations of 
$\{ 1, \dots,\ell\}$.
Let $\partial_1(x_{4i})=x_{4i}$. Then, we have
\begin{align*}
\partial_2 \cdots \partial_\ell (x_I)
& = \sum_{\tau\in S_\ell} (\partial_{\tau(1)}  (x_{4i_1}))\cdots (\partial_{\tau(\ell)}(x_{4i_\ell}))
\\
&= \begin{vmatrix} \partial_1(x_{4i_1}) & \cdots & \partial_1(x_{4i_{\ell}}) 
\\
\vdots & & \vdots \\
\partial_\ell(x_{4i_1}) & \cdots & \partial_\ell( x_{4i_\ell}) 
\end{vmatrix}
\\
&=
\begin{vmatrix}
x_{4i_1} &  \cdots &  x_{4i_\ell} \\
x_{4i_1}^2 & \cdots &  x_{4i_\ell}^2 \\
\vdots & & \vdots \\
x_{4i_1}^{2^{\ell-1}} & \cdots &  x_{4i_\ell}^{2^{\ell-1}}
\end{vmatrix}
\\
 & =e_{\ell}(x_{4i_1}, \dots, x_{4i_\ell})
\end{align*}
as desired.
Similarly, we have 
\[
\partial_2\cdots \partial_{\ell+1} (x_I)=e_{\ell}(x_{4i_1}, \dots, x_{4i_\ell})^2.
\]
Therefore, we have 
\[
\partial_2\cdots \partial_{\ell+2} (x_I)=\partial_{\ell+2}(e_{\ell}(x_{4i_{1}}, \dots, x_{4i_{\ell}})^2)=0. \qedhere
\]
\end{proof}

We end this section by  stating and proving 
Lemma~\ref{lemma:3.3}. We use Lemma~\ref{lemma:3.3} in the computation of $E_{2^{r+2}}$-term 
in Section~\ref{section:5}.

%:  lemma:3.3

\begin{lemma}\label{lemma:3.3}
Let $s$ be a  non-negative integer and $r$ an integer greater than or equal to $2$.
Let $x \in M_{n, \ell-r+1}$ be a linear combination of $\phi_r(x_I)x_3^s$, where 
$\ell(I)=\ell$.
Suppose that $\ell-r+1+s$ is even.
Suppose that $Q_{r+1}(x)=0$.
Then $x$
  is a linear combination of $\phi_{r+1}(x_J)x_3^{s-1}$, where $\ell(J)=\ell+1$.
Moreover, if $s\geq 1$, there exists $y\in M_n$,   which is a linear combination of 
$\phi_r (x_{J}x_3^{s-1})$, such that 
$Q_{r+1}(y)=x$.
\end{lemma}

For $J=\{ j_1, \dots, j_{r}\}$ such that  $2\leq j_1<\cdots < j_{r}$, we denote $\partial_{j_1}\cdots\partial_{j_{r}}$
by $\partial_J$. 
By Proposition~\ref{proposition:3.1}, Lemma~\ref{lemma:3.3} immediately follows from
 Proposition~\ref{proposition:3.4} below.
Lemma~\ref{lemma:3.3} is the case $J=\{2, \dots, r+1 \}$ in Proposition~\ref{proposition:3.4}.

%:  proposition:3.4

\begin{proposition}\label{proposition:3.4}
Suppose that $J=\{ j_1, \dots, j_{r}\}$, $2\leq j_1<\cdots <j_{r}$ and that $i\geq 1$.
Then, the homomorphism \[
\partial_J \colon M_{n,i+r-1} \to M_{n,i}\cap \left( \bigcap_{j\in J} \mathrm{Ker}\, \partial_j\right)
\]
is surjective.
\end{proposition}

\begin{proof}
We prove the proposition by induction on $n$.
First, we deal with the case $n=1$. 
Suppose that $x \in M_{1,1}$. Then,  $x=\alpha x_{41}$
for some  $\alpha\in L_1$. Since
\[
\partial_j(\alpha x_{41})=\alpha x_{41}^{2^{j-1}},
\]
we have 
\[
M_{1,1}\cap \mathrm{Ker}\, \partial_j=\{0\}.
\]
$M_{1,i}=\{0\}$ for $i\geq 2$. So, the proposition for $n=1$ is obvious.

Next, we assume that $n\geq 2$. 
Suppose that $i\geq 1$ and 
\[
f=f_{k,k}x_{4n}^k+f_{k,k+1}x_{4n}^{k+1}+\cdots \in 
M_{n,i} \cap \left( \bigcap_{j\in J} \mathrm{Ker}\, \partial_j \right),
\]
where $f_{k,\ell}\in M_{n-1}$.
We assume for now that there exists an element $g_k \in M_{n-1}$ such that 
$\partial_J(g_k)=f_{k,k}$.
Then, 
\[
f-\partial_J(g_kx_{4n}^k)=f_{k+1,k+1} x_{4n}^{k+1}+f_{k+1,k+2} x_{4n}^{k+2}+\cdots 
\]
So, continuing this process, 
\[
f-\partial_J(g_kx_{4n}^k+g_{k+1}x_{4n}^{k+1}+\cdots )=0
\]

Now, we complete the proof of the proposition by showing the existence of $g_k$ such that $\partial_J(g_k)=f_{k,k}$. 
If $k$ is odd, or if $k$ is even and $i\geq 2$, then
for $j\in J$, 
the coefficient of $x_{4n}^k$ in $\partial_j(f)$ is 
\[
\partial_j(f_{k,k})=0.
\]
It is in $M_{n-1,i}$ if $k$ is even and in $M_{n-1,i-1}$ if $k$ is odd.
So, if $k$ is even or $i\geq 2$, 
by the inductive hypothesis, 
there exists $g_k$ such that $\partial_J(g_k)=f_{k,k}$.
If $k$ is odd and $i=1$, then 
the coefficient of $x_{4n}^{k+2^{j_1-1}-1}$ is
\[
f_{k,k}+\partial_{j_1} f_{k,k+2^{j_1-1}-1} =0.
\]
For $j\in \{ j_2, \dots, j_r\}$, the coefficient of 
$x_{4n}^{k+2^{j_1-1}-1}$ in $\partial_j(f)$ is
\[
\partial_j(f_{k,k+2^{j_1-1}-1})=0.
\]
Since $f_{k,k+2^{j_1-1}-1}$ is in $M_{n-1, 1}$, by the inductive hypothesis, there exists
$g_k$ such that $\partial_{j_2}\cdots \partial_{j_r}(g_k)=f_{k,k+2^{j_1-1}-1}$. So, we have 
$\partial_{J}(g_k)=f_{k,k}$.
This completes the proof.
\end{proof}

\section{The Atiyah-Hirzebruch spectral sequence}\label{section:4}

In this section,  we describe the Atiyah-Hirzebruch spectral sequence 
for 
\[
\bp{m}^*(BG_n)
\]
 for $m\geq 0$ and $n\geq 1$.

We begin with notation. 
Suppose that $y$ ranges over a finite subset $\{ y\}$ of an $R$-module $M$. Let $R'$ be a subring of $R$.
Then, we denote by 
$R'\langle y \rangle$ the $R'$-submodule of $M$ 
generated by  $\{ y\}$. If  $I\subset R'$ is an annihilator ideal of $\{ y\}$, then 
we write 
$(R'/I ) \langle y \rangle$ for $R'\langle y \rangle$.
If it  is a free $(R'/I)$-module, then we write $(R'/I)\{ y\}$ or $(R'/I) \langle y\rangle$ for it. 
If $\{ y \}$ is the empty set, then we put $(R'/I) \{ y \}=(R'/I)\langle y \rangle=\{0\}$.

%:  integral cohomology

Now, we compute 
the integral cohomology $\bp{0}^*(BG_n)=H^{*}(BG_n;\mathbb{Z}_{(2)})$.
We computed the mod $2$ cohomology ring of $BG_n$ in Section~\ref{section:2} and we obtained 
\[
H^{*}(BG_n;\mathbb{Z}/2)=\mathbb{Z}/2[x_2^2,  x_3^2, x_{41}, \dots, x_{4n}] \{ 1, x_2, x_3, x_2 x_3 \}.
\]
The  Bockstein  homomorphism $Q_0$ maps $x_2^2, x_3^2, x_{41}, x_{42}, \dots, x_{4n}$ to zero, so that 
$Q_0$ is a $\mathbb{Z}/2[x_2^2, x_3^2, x_{41}, \dots, x_{4n}]$-module homomorphism. Moreover, since
\[
Q_0(x_2)= x_3, \; Q_0(x_2 x_3) =x_3^2
\]
we have
\begin{align*}
\mathrm{Ker}\, Q_0 & =\mathbb{Z}/2[x_2^2, x_{41}, \dots, x_{4n}] \oplus 
\mathbb{Z}/2[x_2^2, x_3^2, x_{41}, \dots, x_{4n}] \{ x_3, x_3^2  \}, 
\\
\mathrm{Im}\, Q_0 &= \mathbb{Z}/2[x_2^2, x_3^2, x_{41}, \dots, x_{4n}] \{ x_3, x_3^2  \}
\end{align*}
and 
\begin{align*}
\mathrm{Ker}\, Q_0/ \mathrm{Im}\, Q_0 &=\mathbb{Z}/2[x_2^2, x_{41}, \dots, x_{4n}].
\end{align*}
So,  the integral cohomology of $BG_n$ localized at $2$ is 
\[
\mathbb{Z}_{(2)}[x_2^2,  x_{41}, \dots, x_{4n}] \oplus \mathbb{Z}/2[x_2^2, x_3^2, x_{41}, \dots, x_{4n}] \{ x_3, x_3^2 \}.
\]

%:  E_2-term

Next, we describe the $E_2$-term of the Atiyah-Hirzebruch spectral sequence. 
For the sake of notational simplicity, we write $v_0$ for $2$ in the coefficient ring 
$\mathbb{Z}_{(2)}[v_1, \dots, v_m]$
of the truncated Brown-Peterson cohomology $\bp{m}^*(-)$.
Let 
\begin{align*}
D&=\mathbb{Z}_{(2)}[ x_2^2,  x_3^2, x_{41}^2, \dots, x_{4n}^2, v_1, \dots, v_m]/(v_0 x_3^2),
\\
C&=\mathbb{Z}_{(2)}[ x_2^2,   x_{41}^2, \dots, x_{4n}^2, v_1, \dots, v_m].
\end{align*}
We regard $C$ as a subring of $D$.
We write $D_r$ for $D/(v_0, \dots, v_r)$.
We may decompose $D$ to $C\{ 1\} \oplus D_0\{ x_3^2\}$.
As in Section~\ref{section:3}, for a finite set  $I=\{ i_1, \dots, i_\ell \}$ of positive integers $i_1, i_2, \dots, i_\ell$ 
such that $1\leq i_1< \cdots <i_\ell \leq n$, we write $x_I$ for the monomial 
\[
x_{4i_1}\cdots x_{4i_\ell}
\]
and $\ell(I)$ for the integer $\ell\geq 1$.
The number of such sets $I$ is $2^n-1$ and the set $\{ x_{I} \}$ is a finite set consisting 
of $2^n-1$ elements.
Let $y_0$ range over $\{ x_{I} \}$, that is, $\{ y_0\}=\{ x_{I}\}$.
The $E_2$-term of the Atiyah-Hirzebruch spectral sequence for $\bp{m}^*(BG_n)$
is isomorphic to 
\[
D
\{ 1, y_0  \} \oplus D_0 \{ x_3, x_3 y_0 \}=C\{ 1, y_0 \} \oplus D_0\{ x_3, x_3 y_0, x_3^2, x_3^2 y_0\}
\]

%:  d_3 and E_4

Next, we describe the first non-trivial differential $d_3$ and the $E_4$-term.
Let $y_1$, $z_1$ range over 
\begin{align*}
& \{ x_{I} \;|\; \mbox{ $\ell(I)$ is even} \}, 
\\
& \{ x_{I} \;|\; \mbox{ $\ell(I)$ is odd} \},
\end{align*}
respectively.
Then, $\{ y_0\}= \{y_1\} \cup\{z_1\}$. 
Since the Milnor operation $Q_1$ maps $x_3$, $x_{4i}$ to
\[
Q_1( x_3) =  x_3^2, \quad Q_1(x_{4i})=x_3x_{4i}, 
\]
the first non-trivial differential $d_3$ is given by
\begin{align*}
d_3(1) &=0, & d_3(x_3)&= v_1 x_3^2, \\
d_3(y_1) &=0, & d_3(x_3y_1)&=v_1x_3^2 y_1, \\
d_3(z_1) &=v_1x_3 z_1, & d_3(x_3z_1)&=0.
\end{align*}
So, we have
\begin{align*}
\mathrm{Ker}\, d_3&=C\{ 1, y_1, v_0 z_1 \} \oplus D_0\{ x_3^2, x_3^2 y_1, x_3z_1 \},
\\
\mathrm{Im}\, d_3&=D_0 \{  v_1 x_3^2,  v_1 x_3^2 y_1,v_1 x_3 z_1 \} 
\intertext{and}
\mathrm{Ker}\, d_3/\mathrm{Im}\, d_3 &=C \{ 1, y_1,  v_0z_1 \} \oplus D_1\{ x_3^2, x_3^2 y_1,  x_3z_1 \}
.
\end{align*}
One may regard $C\{ 1, y_1, v_0 z_1 \}$, $D_1\{ x_3^2 y_1, x_3 z_1\}$ 
as $C\langle 1, y_1, v_0 y_0 \rangle$, $D_1\langle x_3^2 y_1, x_3z_1\rangle$, respectively. 
Hence, summing up the above, we have the following proposition.

%:  proposition:4.1

\begin{proposition}\label{proposition:4.1}
We have
\begin{align*}
E_2&=D\{1, y_0 \} \oplus D_0 \{ x_3, x_3 y_0  \},
\\
E_4&
=C\langle 1, y_1, v_0y_0 \rangle \oplus D_1 \langle  x_3^2, x_3^2 y_1, x_3z_1\rangle.
\end{align*}
\end{proposition}

%:  higher terms

Now, we describe $E_{2^{r+1}}$-terms for $r\geq 2$.
In Section~\ref{section:3}, for $r\geq 2$, we defined $\phi_r(x_I)$ by
\[
\phi_r(x_I)=\partial_2\cdots \partial_r(x_I), 
\]
in 
\[
M_n=\mathbb{Z}/2[x_2^2, x_3, x_{41}, \dots, x_{4n}].
\]
By abuse of notation, we denote an element in $D \{ 1, y_0\}$ representing 
$\phi_r(x_I)$ in $D_0\{1, y_0\}$  by the same symbol $\phi_r(x_I)$.
We also assume that $e_r$, $y_r$, $z_r$ range over the following finite sets
\begin{align*}
&\{ \phi_r(x_{I}) \;|\; \ell(I)-r+1=0 \}, 
\\
&\{ \phi_r(x_{I})\;|\; \ell(I)-r+1>0 \mbox{ and $\ell(I)-r+1$ is even} \}, 
\\
&\{ \phi_r(x_{I})\;|\; \ell(I)-r+1>0 \mbox{ and $\ell(I)-r+1$ is odd} \},
\end{align*}
respectively. We state our main result  as follows:

%:  proposition:4.2

\begin{proposition}\label{proposition:4.2}
Suppose $r\geq 1$. 
Let $k=\min\{ m, n+1\}$.
For $r\leq k$, the $E_{2^{r+1}}$-term of the Atiyah-Hirzebruch spectral sequence is 
\begin{align*} E_{2^{r+1}} &=
 C\langle 1, y_r, v_sy_s \rangle \oplus D_1/(v_te_t )  \{ x_3^2\} \oplus D_r\langle x_3^2 y_r, x_3 z_r \rangle
\end{align*}
where $s$ ranges over $\{ 0, \dots, r-1\}$ and $t$ ranges over $\{ 2,\dots, r\}$.
Moreover, the spectral sequence collapses at the $E_{2^{k+1}}$-level, that is, $E_\infty=E_{2^{k+1}}$.
\end{proposition}

%:  example:4.3

\begin{example}
\label{example:4.3}
We compute the Brown-Peterson cohomology of $BG_1$.
\\
The case $n=1$, $m\geq 2$:
\[
\begin{array}{lll}
\{ y_1\}=\emptyset, & \{ z_1\}=\{ x_{41} \}, \\
\{ y_2 \}=\emptyset, & \{ z_2\}= \emptyset, & \{ e_2\}=\{ x_{41}^2\}.
\end{array}
\]
\begin{align*}
E_2&=C\{ 1, x_{41}\} \oplus D_0 \{ x_{3}^2, x_{3}^2 x_{41}, x_3, x_3 x_{41} \}\\
E_4&=C\{ 1, v_0 x_{41}\} \oplus D_1 \{ x_3^2\} \oplus D_1\{ x_3 x_{41} \}, \\
E_8& =C\{ 1 , v_0 x_{41} \} \oplus D_1/(v_2 x_{41}^2) \{ x_3^2\}
\end{align*} 
So, $x_{41}$ lifts up to $\bp{0}^{*}(BG_1)$ that corresponds to the $E_2$-term above 
and it does not lift to $\bp{1}^{*}(BG_1)$ that corresponds to the $E_4$-term above.
\end{example}

%:  example:4.4

\begin{example}
\label{example:4.4}
We compute the Brown-Peterson cohomology of $BG_2$.
The case $n=2$, $m\geq 3$:
\[
\begin{array}{lll}
\{ y_1\}=\{ x_{41}x_{42} \}, & \{ z_1\}=\{ x_{41}, x_{42} \}, \\
\{ y_2 \}= \emptyset, &  \{ z_2\}=\{ x_{41}^2x_{42}+x_{41}x_{42}^2  \}, & \{ e_2\}=\{ x_{41}^2, x_{42}^2 \}, \\
\{ y_3\}=\emptyset, & \{ z_3\}=\emptyset, & \{ e_3\}=\{ x_{41}^2x_{42}^4+x_{41}^4x_{42}^2 \}.
\end{array}
\]
\begin{align*}
E_4 & = C\{ 1, x_{41}x_{42}, v_0 x_{41}, v_0x_{42} \}
\\
& \oplus D_1\{ x_{3}^2\} \oplus D_1\{ x_{41}x_{42} x_{3}^2, x_{3}x_{41}, x_{3}x_{42} \}
\\
E_8& =
C\langle 1, v_1 x_{41}x_{42}, v_0 x_{41}x_{42}, v_0 x_{41}, v_0 x_{42}\rangle \\
& \oplus D_1/(v_2x_{41}^2, v_2 x_{42}^2) \{ x_3^2\}  \oplus D_2\{ x_{3} (x_{41}^2x_{42}+x_{41}x_{42}^2)\}
\\
E_{16}&=C\langle 1, v_1 x_{41}x_{42}, v_0x_{41}x_{42}, v_0 x_{41}, v_0 x_{42} \rangle 
\\
& \oplus
D_1/(v_2 x_{41}^2, v_2 x_{42}^2, v_3 (x_{41}^4x_{42}^2+x_{41}^2 x_{42}^4) )\{ x_{3}^2\}
\end{align*}
So, $x_{41}x_{42}$ in $\bp{1}^*(BG_2)$ that corresponds to the $E_4$-term above does not lift up
 to $\bp{2}^{*}(BG_2)$ that corresponds to the $E_8$-term above.
\end{example}

%:  proof of theorem:1.1

We end this section by proving Theorems~\ref{theorem:1.1} and \ref{theorem:1.4}.
These are immediate from Proposition~\ref{proposition:4.2}.

\begin{proof}[Proof of Theorem~\ref{theorem:1.1}]
Given $m$, let $n=m+1$.
Then, the Atiyah-Hirzebruch spectral sequence for $\bp{m}^{*}(BG_{m+1})$ collapses at the 
$E_{2^{m+1}}$-level and
\[
E_\infty= C\langle 1, y_m, v_sy_s \rangle \oplus D_1/(v_te_t )  \{ x_3^2\} \oplus 
D_m\langle x_3^2 y_m, x_3 z_m \rangle.
\]
Let $u$ be an element in $\bp{m}^*(BG_{m+1})$ representing  the permanent cycle 
\[
\phi_{m}(x_{41}\cdots x_{4,m+1})\in \{ y_{m} \}
\]
 in the Atiyah-Hirzebruch spectral sequence. 
 It is non-zero by Lemma~\ref{ lemma:3.2}.
 Moreover, since the corresponding element in the Atiyah-Hirzebruch spectral sequence for 
 $\bp{m+1}^*(BG_{m+1})$ supports non-trivial differential and so $u$ does not lift up to 
 $\bp{m+1}^*(BG_{m+1})$.
Moreover, the reduction $\rho^m_0(u)$ is non-zero in  
\[
\mathbb{Z}_{(2)}[x_2^2, x_{41}, \dots, x_{4, m+1}] = H^{*}(BG_{m+1}; \mathbb{Z}_{(2)})/(x_3), 
\]
therefore it is non-torsion.
\end{proof}

%:  proof of theorem:1.4

\begin{proof}[Proof of Theorem~\ref{theorem:1.4}]
By Proposition~\ref{proposition:4.2}, the odd degree part of the $E_\infty$-term is 
\[
D_r\langle x_3z_r\rangle.
\]
However, by definition, 
$\{ z_r\}=\emptyset$ if $r \geq n+1$.
By Lemma~\ref{ lemma:3.2},
$\{ z_r\}$ contains non-zero element $\phi_r(x_{41}\cdots x_{4r})$ for $r\leq n$. Hence, 
$D_r\langle x_3z_r\rangle=\{0\}$ if and only if $r\geq n+1$.
If $m\leq n$, the spectral sequence collapses at the $E_{2^{m+1}}$-level and it contains 
non-zero odd degree elements.
If $m\geq n+1$, then $E_{2^{n+2}}$-term has no non-zero odd degree element. 
So, the spectral sequence collapses at the $E_{2^{n+2}}$-level
and the $E_\infty$-term also does not have any non-zero odd degree element.
\end{proof}

\section{Proof of Proposition~\ref{proposition:4.2}}\label{section:5}

To prove Proposition~\ref{proposition:4.2}, 
we need to study non-trivial differentials in the spectral sequence. To this end, we prove
Lemmas~\ref{lemma:5.1} and  \ref{lemma:5.2} before the proof of Proposition~\ref{proposition:4.2}.

We use the following Lemma~\ref{lemma:5.1} to show that the differential $d_u$ is trivial for 
$2^{r}\leq u \leq 2^{r+1}-2$. 

%:  lemma:5.1

\begin{lemma}\label{lemma:5.1}
The following statements {\rm (1)}, {\rm (2)} and {\rm (3)} hold:
\begin{itemize}
\item[(1)]
For $2 \leq u < 2^{r+2}-1$, 
$(D_r\langle x_3^2 y_r, x_3 z_r\rangle)^{*,-u+1}=\{ 0\}$.
\item[(2)]
Suppose that $x\in C\langle 1, y_r, v_s y_s \rangle$ and 
$
v_2 x=\cdots =v_r x=0.
$
Then, we have $x=0$.
\item[(3)]
Suppose that $x\in D_1/(v_t e_t)\{ x_3^2\}$, 
$x\equiv 0$ in $D_r/(v_t e_t)\{ x_3^2\}=D_r\{ x_3^2\}$ and  
$
v_2 x=\cdots =v_r x=0.
$
 Then, we have
$x=0$ in $D_1/(v_t e_t)\{ x_3^2\}$.
\end{itemize}
\end{lemma}
\begin{proof} 
Statement (1) is clear from the fact that $D_r^{*,-u+1}=\{0\}$ for $2 \leq u <2^{r+2}-1$.
Statement (2) is immediate from the fact that
$
C\langle y_r, v_s y_s \rangle
$
is a $C$-submodule of the free $C$-module $C\{ y_0 \}$.
In order to prove statement (3), we consider 
obvious projection
\[
\pi^{i}_{j}:D_i/(v_t e_t) \{ x_3^2\} \to D_j/(v_t e_t ) \{ x_3^2\},
\]
where $1\leq i\leq j \leq r$.
Since $t\leq r$, it is clear that 
\[
D_r/(v_t e_t)=D_r=\mathbb{Z}/2[x_2^2, x_3^2, x_{41}^2, \dots, x_{4n}^2, v_{r+1}, \dots, v_m]. 
\]
By the assumption, we have $\pi^1_{r}(x)=0$.
Moreover, for $i=1, \dots, r-1$, in $D_i$, $(v_t e_t)=( v_{i+1} e_{i+1}, \dots, v_r e_r)$
and
the kernel of 
$\pi^{i}_{i+1}$ is 
\[
\mathrm{Ker}\, \pi^{i}_{i+1}=(v_{i+1} x_3^2)= D_{i}/(e_{i+1}, v_{i+2} e_{i+2}, \dots, v_r e_r)\{ v_{i+1} x_3^2\}
\]
and $\mathrm{Ker}\, \pi^{i}_{i+1}$ is a free $\mathbb{Z}/p[ v_{i+1} ]$-module. 
Therefore, $v_{i+1}  \pi^1_i(x)=0$ implies $\pi^1_i(x)=0$.
Thus, we have $\pi^1_{i}(x)=0$ for $i=r-1, \dots, 2, 1$ and we have the desired result.
\end{proof}

%:  Morava K-theory

We use the Morava K-theory in the proof of Proposition~\ref{proposition:4.2}. So, we briefly recall it.
The Morava K-theory $K(r+1)^*(X)$ of a space $X$ is a generalized cohomology theory whose coefficient ring is 
$\mathbb{Z}/2[v_{r+1}, v_{r+1}^{-1}]$
and there exists a natural transformation $\psi_{r+1}\colon \bp{r+1}^*(X)\to K(r+1)^*(X)$
sending $v_i$ to $0$ for $i\leq r$ and $v_{r+1}$ to $v_{r+1}$.
The first and the only non-trivial differential of the Atiyah-Hirzebruch spectral sequence 
for $K(r+1)^{*}(B\mathbb{Z}/2)$ is
\[
d_{2^{r+2}-1} (x_1)=v_{r+1}x_1^{2^{r+2}}=v_{r+1}Q_{r+1}(x_1)
\]
where $x_1$ is the generator of $H^1(B\mathbb{Z}/2;\mathbb{Z}/2)$. 
So, by the K\"{u}nneth formula, we have Lemma~\ref{lemma:5.2} below.

%:  lemma:5.2

\begin{lemma} \label{lemma:5.2}
The first nontrivial differential of the Atiyah-Hirzebruch spectral sequence 
for $K(r+1)^*(B(\mathbb{Z}/2)^{n+2})$ is 
\[
d_{2^{r+2}-1}(x)=v_{r+1} Q_{r+1}(x)
\]
\
for all $x$ in $H^{*}(B(\mathbb{Z}/2)^{n+2};\mathbb{Z}/2)=E_2^{*,0}$ where $E_2^{*,*}$ is the $E_2$-term of the 
Atiyah-Hirzebruch spectral sequence for the Morava $K$-theory $K(r+1)^{*}(B(\mathbb{Z}/2)^{n+2})$.
\end{lemma}
We use Lemma~\ref{lemma:5.2} to determine the differential $d_{2^{r+2}-1}$.

Now, we prove Proposition~\ref{proposition:4.2}.

%:  proof of proposition:4.2

\begin{proof}[Proof of Proposition~\ref{proposition:4.2}]
We prove Proposition~\ref{proposition:4.2} by induction on $r\geq 1$.

For $r=1$, by Proposition~\ref{proposition:4.1}, the proposition holds.

Assume that $u\geq 2^{r+1}$ and 
\[
E_u=E_{2^{r+1}}=C\langle 1, y_r, v_{s} y_{s} \rangle \oplus
D_1/(v_t e_t) \{ x_3^2 \} \oplus 
D_r \langle  x_3^2 y_r, x_{3 }z_{r} \rangle.
\]
Let us write 
$E^{even,*}$, $E^{odd,*}$ for 
\[
 C\langle 1, y_r, v_{s} y_{s} \rangle \oplus
D_1/(v_t e_t) \{ x_3^2 \} \oplus 
D_r \langle  x_3^2 y_r \rangle, 
\quad D_r \langle  x_3 z_r \rangle, 
\]
respectively.
We prove  $d_u=0$ for $u=2^{r+1}, \dots, 2^{r+2}-2$.
Since $\bp{m}^*$ is concentrated in even degrees, $d_u=0$ if $u$ is even. So, we assume that $u$ is odd.
Then, we have
\begin{align*}
d_u(E^{even,*} )& \subset E^{odd,*},
\\
d_u(E^{odd,*} )& \subset E^{even,*}.
\end{align*}
First, we prove $d_{u}=0$ on $E^{even, *}$.
By Lemma~\ref{lemma:5.1} (1),  
we have 
\[
d_{u}(E^{even,0})\subset (D_r \langle  x_3 z_r \rangle)^{*,-u+1}=\{ 0\}.
\]
Since $d_u$ is a $\bp{m}^*$-module homomorphism and $D_1/(v_t e_t) \{ x_3^2 \} \oplus 
D_r \langle  x_3^2 y_r \rangle$ is generated by 
\[
(D_1/(v_t e_t) \{ x_3^2 \} \oplus 
D_r \langle  x_3^2 y_r \rangle) \cap E^{even,0}
\]
 as a $\bp{m}^*$-module, we have 
\[
d_u(D_1/(v_t e_t) \{ x_3^2 \} \oplus 
D_r \langle  x_3^2 y_r \rangle )=\{ 0\}. 
\]
On the other hand, on $D_r \langle  x_3 z_r \rangle$, the multiplication by $x_3^2$ is injective.
Therefore, since 
\begin{align*}
x_3^2 d_u(C\langle 1, y_r, v_s y_s \rangle)& = d_u(x^2 C\langle 1, y_r, v_s y_s \rangle)
\\
&\subset  d_u(D_1/(v_t e_t) \{ x_3^2 \} \oplus 
D_r \langle  x_3^2 y_r \rangle )=\{ 0\},
\end{align*}
we have 
\[
d_u(C\langle 1, y_r, v_s y_s \rangle)=\{ 0\}
\]
and so 
$d_u(E^{even,*})=\{ 0\}$.
Next, we prove $d_u=0$ on $E^{odd,*}$. 
By Lemma~\ref{lemma:5.1} (2), (3), we have $d_u(E^{odd,*})\subset D_r\langle x_3^2y_r \rangle$
and by Lemma~\ref{lemma:5.1} (1), $d_u(E^{odd,0})\subset  (D_r\langle x_3^2 y_r \rangle)^{*,-u+1}=\{0\}$.
Since $E^{odd,*}$ is generated by $E^{odd,0}$ as a $\bp{m}^*$-module, we have 
$d_u(E^{odd,*})=\{0\}$. 
Thus, $d_u$ is trivial for $2^{r+1}\leq u \leq 2^{r+2}-2$.

Next, we compute $d_{2^{r+2}-1}$ for $r\leq k-1$.
We consider homomorphisms 
\begin{align*}
& \rho^{m}_{r+1}\colon \bp{m}^*(BG_n) \to \bp{r+1}^*(BG_n),
\\
&\iota_{n}^* \colon \bp{r+1}^*(BG_n) \to \bp{r+1}^*(B(\mathbb{Z}/2)^{n+2}),
\\
&\psi_{r+1}\colon \bp{r+1}^*(B(\mathbb{Z}/2)^{n+2}) \to  K(r+1)^{*}(B(\mathbb{Z}/2)^{n+2}).
\end{align*}
We denote the induced homomorphisms among the Atiyah-Hirzebruch spectral sequences by the same symbols.
Note that $\deg v_{r+2}=-2^{r+3}+2$.
Then,  the induced homomorphism $\rho^m_{r+1}$ between 
Atiyah-Hirzebruch spectral sequences for $\bp{m}^*(BG_n)$ and
$\bp{r+1}^*(BG_n)$ is an isomorphism on $E_{2^{r+2}-1}^{*,w}$ 
for $-2^{r+3}+2 < w$.
A nonzero element in 
$D_r\langle x_3^2y_r, x_3 z_r\rangle$ of the 
$E_{2^{r+2}-1}$-term of the Atiyah-Hirzebruch spectral sequence for $\bp{r+1}^{*}(BG_n)$ is of the form 
\[
v_{r+1}^{k} f
\]
where $f$ is a nonzero element in $H^{*}(BG_n)$.
The $E_{2^{r+2}-1}$-term of the Atiyah-Hirzebruch spectral sequence for $K(r+1)^*(B(\mathbb{Z}/2)^{n+2})$ 
is equal to its $E_2$-term
\[
\mathbb{Z}/2[v_{r+1}, v_{r+1}^{-1}] \otimes H^{*}(B(\mathbb{Z}/2)^{n+2}).
\]
The induced homomorphism $\psi_{r+1}\circ \iota_n^*$ sends 
$v_{r+1}^k f$ to $v_{r+1}^k \iota_n^*(f)$. 
By Proposition~\ref{proposition:2.2}, $\iota_n^{*}\colon H^{*}(BG_n)\to H^{*}(B(\mathbb{Z}/2)^{n+2})$ is injective.
So, the induced homomorphism  $\psi_{r+1}\circ \iota_n^{*}$ is injective on $D_r\langle x_3^2y_r, x_3 z_r\rangle$.
Similarly, the kernel of $\psi_{r+1}\circ \iota_n^*$
on $D_1/(v_te_t)\{x_3^2\}$ of the $E_{2^{r+2}-1}$-term of the Atiyah-Hirzebruch spectral sequence for $\bp{r+1}^{*}(BG_n)$
  is the kernel of the obvious projection 
\[
\pi^1_r\colon D_1/(v_te_t)\{ x_3^2\} \to D_r\{ x_3^2\}.
\]
By Lemma~\ref{lemma:5.2}, the first non-trivial differential in the Atiyah-Hirzebruch spectral sequence 
for the Morava $K$-theory $K(r+1)^*(B(\mathbb{Z}/2)^{n+2})$ is 
\[
d_{2^{r+2}-1}(x)=v_{r+1}Q_{r+1}(x) \in E_{2^{r+2}-1}^{*, -2^{r+2}+2}
\]
for $x\in E_{2^{r+2}-1}^{*,0}$.
In the Atiyah-Hirzebruch spectral sequence for $\bp{m}^*(BG_n)$, we have 
\[
d_{2^{r+2}-1}(E^{even,0})\subset (D_r\langle x_3z_r\rangle)^{*,-2^{r+2}+2}
\]
and
\[
d_{2^{r+2}-1}(E^{odd,0}) \subset (D_1/(v_te_t)\{ x_3^2\} \oplus D_r\langle x_3^2 y_r\rangle)^{*,-2^{r+2}+2}.
\]
Hence, we have 
\begin{align*}
d_{2^{r+2}-1}(\phi_r(x_I))&=v_{r+1} \phi_{r+1}(x_I) x_3, \\
d_{2^{r+2}-1}(\phi_r(x_I)x_3)&=v_{r+1} \phi_{r+1}(x_I) x_3^2+\alpha_I,
\end{align*}
where $\alpha_I$ is in $\mathrm{Ker}\, \pi^1_r$.
However, since 
\[
v_2 \phi_r(x_I)x_3=\cdots =v_r\phi_r(x_I)x_3=0
\]
 and 
 \[
 v_2 \phi_{r+1}(x_I)x_3^2=\cdots =v_r\phi_{r+1}(x_I)x_3^2=0, 
 \]
we have
$v_2\alpha_I=\cdots =v_r \alpha_I=0$. By Lemma~\ref{lemma:5.1} (3), we have $\alpha_I=0$.
Therefore, we have 
\begin{align*}
d_{2^{r+2}-1}(\phi_r(x_I))&=v_{r+1} \phi_{r+1}(x_I) x_3, \\
d_{2^{r+2}-1}(\phi_r(x_I)x_3)&=v_{r+1} \phi_{r+1}(x_I) x_3^2.
\end{align*}

Now, we compute the $E_{2^{r+2}}$-terms.
We have that 
\[
\mathrm{Im}\, d_{2^{r+2}-1}=  (v_{r+1}e_{r+1}) \{x_3^2\} \oplus D_r \langle v_{r+1} x_3^2 y_{r+1}, v_{r+1} x_3 z_{r+1}  \rangle 
,
\]
where $(v_{r+1}e_{r+1})$ is the ideal generated by $v_{r+1}e_{r+1}$ in $D_1/(v_te_t)$ and  $t$ ranges over 
$\{ 2, \dots, r\}$.
The kernel of the multiplication by $x_3^2$ is  $C\langle v_1, v_2, \dots, v_r, v_{r}y_r , v_s y_s\rangle$.
On the other hand, by Lemma~\ref{lemma:3.3}, we have
\begin{align*}
 \mathrm{Ker}\, d_{2^{r+1}-1} \cap \left( D_1/(v_te_t )  
 \{ x_3^2\} \oplus D_r\langle x_3^2 y_r, x_3 z_r \rangle \right) \\
= 
D_1/(v_t e_t)\{ x_3^2 \}\oplus D_r \langle  x_{3}^2 y_{r+1}, x_3 z_{r+1},  \rangle.
\end{align*}
Hence, we have the required $E_{2^{r+2}}$-term.

Finally, we complete the proof by showing $d_u=0$ for $u\geq 2^{k+1}$.
Since $k=\min\{ m, n+1\}$, we need to consider two cases:
\\
Case 1. The case $k=m$. In this case, $r=k$, $u\geq 2^{r+1}$.
Then, we have 
\[
E_u=E^{even,*}\oplus E^{odd,*}.
\]
Moreover, we have $D_r=\mathbb{Z}/2$.
On the one hand,  we have
\[
(D_r\{ x_3^2\})^{*, -u+1}=\{0\}.
\] 
So, by Lemma~\ref{lemma:5.1} (2), (3), 
we have $d_u(E^{odd,*})\subset D_r\langle x_3^2 y_r\rangle$.
It is clear that $d_u(E^{even,*})\subset D_r\langle x_3 z_r\rangle$.
On the other hand, from $D_r=\mathbb{Z}/2$, we have
\[
(D_r\langle x_3^2 y_r, x_3 z_r \rangle)^{*,-u+1} =\{0\}.
\]
Therefore, we  have
\begin{align*}
d_u(E_u) \subset (D_r\langle x_3^2 y_r , x_3 z_r \rangle)^{*,-u+1} =\{0\}.
\end{align*}
\\
Case 2. 
The case $k=n+1$. In this case, as in the proof of Theorem~\ref{theorem:1.4}, since $\{ z_k\} =\emptyset$, 
the $E_{2^{k+1}}$-term has no odd degree generators. Therefore, 
the spectral sequence collapses at the $E_{2^{k+1}}$-level.
\end{proof}

{\it Acknowledgement.}
The author would like to thank the anonymous referee for his or her advice.
 It was helpful in improving the presentation of this paper.

%:  references

%\bibliographystyle{halpha} 
%\bibliography{bib}

\begin{bibdiv}



\begin{biblist}

\bib{atiyah-hirzebruch-1962}{article}{
   author={Atiyah, M. F.},
   author={Hirzebruch, F.},
   title={Analytic cycles on complex manifolds},
   journal={Topology},
   volume={1},
   date={1962},
   pages={25--45},
   issn={0040-9383},
   review={\MR{0145560}},
   doi={10.1016/0040-9383(62)90094-0},
}

\bib{antieau-2016}{article}{
   author={Antieau, Benjamin},
   title={On the integral Tate conjecture for finite fields and
   representation theory},
   journal={Algebr. Geom.},
   volume={3},
   date={2016},
   number={2},
   pages={138--149},
   issn={2214-2584},
   review={\MR{3477951}},
   doi={10.14231/AG-2016-007},
}

\bib{colliot-thelene-szamuely-2010}{article}{
   author={Colliot-Th\'el\`ene, Jean-Louis},
   author={Szamuely, Tam\'as},
   title={Autour de la conjecture de Tate \`a coefficients ${\bf Z}_{\ell}$
   pour les vari\'et\'es sur les corps finis},
   language={French},
   conference={
      title={The geometry of algebraic cycles},
   },
   book={
      series={Clay Math. Proc.},
      volume={9},
      publisher={Amer. Math. Soc., Providence, RI},
   },
   date={2010},
   pages={83--98},
   review={\MR{2648666}},
}



\bib{kameko-2015}{article}{
   author={Kameko, Masaki},
   title={On the integral Tate conjecture over finite fields},
   journal={Math. Proc. Cambridge Philos. Soc.},
   volume={158},
   date={2015},
   number={3},
   pages={531--546},
   issn={0305-0041},
   review={\MR{3335426}},
   doi={10.1017/S0305004115000134},
}
		
\bib{kono-yagita-1993}{article}{
   author={Kono, Akira},
   author={Yagita, Nobuaki},
   title={Brown-Peterson and ordinary cohomology theories of classifying
   spaces for compact Lie groups},
   journal={Trans. Amer. Math. Soc.},
   volume={339},
   date={1993},
   number={2},
   pages={781--798},
   issn={0002-9947},
   review={\MR{1139493}},
   doi={10.2307/2154298},
}
		

\bib{landweber-1972}{article}{
   author={Landweber, Peter S.},
   title={Elements of infinite filtration in complex cobordism},
   journal={Math. Scand.},
   volume={30},
   date={1972},
   pages={223--226},
   issn={0025-5521},
   review={\MR{0326762}},
   doi={10.7146/math.scand.a-11077},
}

\bib{pirutka-yagita-2015}{article}{
   author={Pirutka, Alena},
   author={Yagita, Nobuaki},
   title={Note on the counterexamples for the integral Tate conjecture over
   finite fields},
   journal={Doc. Math.},
   date={2015},
   number={Extra vol.: Alexander S. Merkurjev's sixtieth birthday},
   pages={501--511},
   issn={1431-0635},
   review={\MR{3404393}},
}

\bib{quick-2016}{article}{
	author = {Quick, Gereon},
        eprint = {arXiv:1611.09136},
	title = {Examples of non-algebraic classes in the Brown-Peterson tower},
	date = {2016},
}	
	
\bib{soule-voisin-2005}{article}{
   author={Soul\'e, C.},
   author={Voisin, C.},
   title={Torsion cohomology classes and algebraic cycles on complex
   projective manifolds},
   journal={Adv. Math.},
   volume={198},
   date={2005},
   number={1},
   pages={107--127},
   issn={0001-8708},
   review={\MR{2183252}},
   doi={10.1016/j.aim.2004.10.022},
}

\bib{totaro-1997}{article}{
   author={Totaro, Burt},
   title={Torsion algebraic cycles and complex cobordism},
   journal={J. Amer. Math. Soc.},
   volume={10},
   date={1997},
   number={2},
   pages={467--493},
   issn={0894-0347},
   review={\MR{1423033}},
   doi={10.1090/S0894-0347-97-00232-4},
}

%\bib{totaro-2014}{book}{
%   author={Totaro, Burt},
%   title={Group cohomology and algebraic cycles},
%   series={Cambridge Tracts in Mathematics},
%   volume={204},
%   publisher={Cambridge University Press, Cambridge},
%   date={2014},
%   pages={xvi+228},
%   isbn={978-1-107-01577-7},
%   review={\MR{3185743}},
%   doi={10.1017/CBO9781139059480},
%}

\bib{yagita-2005}{article}{
   author={Yagita, Nobuaki},
   title={Applications of Atiyah-Hirzebruch spectral sequences for motivic
   cobordism},
   journal={Proc. London Math. Soc. (3)},
   volume={90},
   date={2005},
   number={3},
   pages={783--816},
   issn={0024-6115},
   review={\MR{2137831}},
   doi={10.1112/S0024611504015084},
}

	\end{biblist}
	
	\end{bibdiv}

%: references

\end{document}